\documentclass[11pt, letterpaper]{amsart}
\usepackage{graphicx, amssymb, color, hyperref}

\newtheorem{thm}{Theorem}[section]

\newtheorem{lem}[thm]{Lemma}
\newtheorem{prop}[thm]{Proposition}

\theoremstyle{definition}
\newtheorem{defn}[thm]{Definition}

\theoremstyle{remark}
\newtheorem{rem}[thm]{Remark}

\numberwithin{equation}{section}

\newcommand{\abs}[1]{\left\vert#1\right\vert}
\newcommand{\set}[1]{\left\{#1\right\}}
\newcommand{\Real}{\mathbb R}

\newcommand{\Natural}{\mathbb N}
\newcommand{\nin}{n \in \Natural}
\newcommand{\kin}{k \in \Natural}

\newcommand{\such}{{\ | \ }}
\newcommand{\limn}{\lim_{n \to \infty}}

\newcommand{\limk}{\lim_{k \to \infty}}
\newcommand{\dfn}{\, := \,}

\newcommand{\prob}{\mathbb{P}}

\newcommand{\conv}{\mathsc{conv}}

\newcommand{\qprob}{\mathbb{Q}}
\newcommand{\expec}{\mathbb{E}}
\newcommand{\expecp}{\expec_\prob}
\newcommand{\expecq}{\expec_\qprob}

\newcommand{\Lb}{\mathbb{L}}
\newcommand{\lz}{\Lb^0}

\newcommand{\lzp}{\lz_{+}}

\newcommand{\M}{\mathbb{M}}
\newcommand{\Mp}{\M^0_+}

\newcommand{\F}{\mathcal{F}}

\newcommand{\ud}{\mathrm d}

\newcommand{\inner}[2]{\left \langle #1 , \, #2 \right \rangle}

\newcommand{\num}{num\'eraire}
\newcommand{\X}{\mathcal{X}}
\newcommand{\oX}{\overline{\X}}

\newcommand{\oO}{\overline{\mathcal{O}}}
\newcommand{\On}{\mathcal{O}^{\X}_n}
\newcommand{\oOn}{\overline{\mathcal{O}^{\X}_n}}
\newcommand{\Oy}{\mathcal{O}^{\X}_Y}
\newcommand{\Ony}{\mathcal{O}^{\X}_{nY}}

\newcommand{\soco}{\mathsf{sc}}

\newcommand{\So}{\mathcal{S}}

\newcommand{\oS}{\overline{\mathcal{S}}}

\newcommand{\eL}{\mathcal{P}}
\newcommand{\oL}{\overline{\eL}}

\newcommand{\pare}[1]{\left(#1\right)}
\newcommand{\bra}[1]{\left[#1\right]}
\newcommand{\dbra}[1]{[\kern-0.15em[ #1 ]\kern-0.15em]}
\newcommand{\dbraco}[1]{[\kern-0.15em[ #1 [\kern-0.15em[}
\newcommand{\dbraoc}[1]{]\kern-0.15em] #1 ]\kern-0.15em]}
\newcommand{\C}{\mathcal{C}}
\newcommand{\K}{\mathcal{O}}
\newcommand{\Ko}{\mathcal{K}}

\newcommand{\D}{\mathcal{D}}

\newcommand{\indic}{\mathbb{I}}

\newcommand{\absco}{{<\kern-0.53em<}}

\newcommand{\oC}{\overline{\C}}
\newcommand{\oK}{\overline{\K}}

\newcommand{\oD}{\overline{\D}}

\renewcommand{\conv}{\mathsf{conv}}

\begin{document}

\title[Uniform integrability and local convexity in $\lz$]{Uniform integrability and local convexity in $\lz$}%
\author{Constantinos Kardaras}%
\address{Constantinos Kardaras, Department of Statistics, London School of Economics and Political Science, 10 Houghton st, London, WC2A 2AE, UK.}%
\email{K.Kardaras@lse.ac.uk}%

\thanks{The author would like to thank Gordan {\v{Z}}itkovi{\'{c}} for fruitful discussions regarding this work.}%

\date{\today}%
\begin{abstract}
Let $\lz$ be the vector space of all (equivalence classes of) real-valued random variables built over a probability space $(\Omega, \F, \prob)$, equipped with a metric topology compatible with convergence in probability. In this work, we provide a necessary and sufficient structural condition that a set $\X \subseteq \lz$ should satisfy in order to infer the existence of a probability $\qprob$ that is equivalent to $\prob$ and such that $\X$ is uniformly $\qprob$-integrable. Furthermore, we connect the previous essentially measure-free version of uniform integrability with local convexity of the $\lz$-topology when restricted on convex, solid and bounded subsets of $\lz$.
\end{abstract}

\maketitle


\section*{Introduction}

In the study of probability measure spaces, the notion of uniform integrability for sets of integrable measurable functions (random variables) has proved essential in both fields of Functional Analysis and Probability. On one hand, the Dunford-Pettis theorem \cite[Chapter IV.9]{MR1009162} states that uniform integrability of a set of random variables is equivalent to its relative weak sequential compactness in the corresponding $\Lb^1$ space; this fact allows the utilization of powerful functional-analytic techniques. On the other hand, uniform integrability is exactly the extra condition needed in order for convergence in (probability) measure to imply convergence in the $\Lb^1$-norm---see, for example, \cite[Proposition 4.12]{MR1876169}.

The latter fact mentioned above has a simple, yet important, corollary in the topological study of $\lz$ spaces, where $\lz$ is defined as the set of all (equivalence classes of) real-valued random variables built over a probability space $(\Omega, \F, \prob)$, equipped with a metric topology compatible with convergence in probability. To wit, when a uniformly $\prob$-integrable set $\X \subseteq \Lb^1 (\prob)$ is regarded as a subset in $\lz$, the $\lz$-topology and $\Lb^1(\prob)$-topology on $\X$ coincide; in particular, the $\lz$-topology on $\X$ is locally convex whenever $\X$ is uniformly $\prob$-integrable. Local convexity of the considered topology  is an important property from a functional-analytic viewpoint, as it enables the use of almost indispensable machinery, such as the Hahn-Banach Theorem and its consequences. Unfortunately, even though $\lz$ constitutes a very natural modelling environment (for example, it is the only of the $\Lb^p$ spaces for $p \in [0, \infty)$ that remains invariant with respect to equivalent changes of probability measure), the convex-analytic structure of $\lz$ is quite barren. Indeed, when the underlying probability space is non-atomic, the topological dual of $\lz$ contains only the zero functional; furthermore, unless the underlying  probability space is purely atomic, the $\lz$-topology fails to be locally convex---for the previous facts, see \cite[Theorem 13.41]{MR2378491}.

Despite the ``hostility'' of its topological environment, considerable research has been carried out in order to understand the convex-analytic properties of $\lz$---for a small representative list, see \cite{Kom67}, \cite{MR1768009}, \cite{MR1304434}, \cite{Zit09}, \cite{Kar10a}, \cite{MR2823052}. In the spirit of the discussion in the previous paragraphs, a novel use of uniform integrability bridging Functional Analysis and Probability was recently provided in \cite{KarZit12}. Consider an $\lz$-convergent sequence $(X_n)_{\nin}$ of random variables in $\lzp$ (the latter denoting the non-negative orthant of $\lz$), and define $\X$ as the $\lz$-closure of $\conv \pare{\set{X_n \such \nin}}$, where ``$\conv$'' is used to denote the convex hull of a subset in $\lz$. One of the main messages of \cite{KarZit12} is that the restriction of the $\lz$-topology on $\X$ is locally convex if and only if there exists \emph{some} probability $\qprob \sim \prob$ (where ``$\sim$'' is used to denote equivalence of probability measures) such that $\X$ is uniformly $\qprob$-integrable. Loosely speaking, the fact that the restriction of the $\lz$-topology on $\X$ is locally convex can be regarded as an essentially measure-free version of uniform integrability.

In the present work, the previous topic is explored in greater depth. The first main result of the paper provides a structural necessary and sufficient condition for an \emph{arbitrary} subset $\X$ of $\lz$ to be uniformly $\qprob$-integrable under some $\qprob \sim \prob$. To be more precise, for all $\nin$ define $\On$ as the subset of $\lzp$ consisting of all random variables dominated in the lattice structure of $\lz$ by some random variable of $\conv \pare{\set{(|X| - n)_+ \such X \in \X}}$, where $Z_+ = \max \set{Z, 0}$ for $Z \in \lz$; then, there is equivalence between the condition $\bigcap_{\nin} \On = \set{0}$ and existence of $\qprob \sim \prob$ such that $\X$ is uniformly $\qprob$-integrable. When $\X \subseteq \lzp$, the structural condition $\bigcap_{\nin} \On = \set{0}$ has a useful interpretation in the field of Financial Mathematics: its failure implies that there exists $Y \in \lzp \setminus \set{0}$ with the property that, for all $\nin$, there is the possibility of super-hedging $Y$ using convex combinations of call options to exchange positions in $\X$ for $n$ units of cash. The second main result of the paper explores further the connection between local convexity of the $\lz$-topology and the previous essentially measure-free version of uniform integrability. For $\X \subseteq \lzp$ that is convex, $\lz$-bounded and solid (in the sense that $X \in \X$ and $Y \in \lzp$ with $Y \leq X$ implies $Y \in \X$), it is established that the restriction of the $\lz$-topology on $\X$ is locally convex if and only if there exists $\qprob \sim \prob$ such that $\X$ is uniformly $\qprob$-integrable. The previous result sheds an important light on the topological structure of convex subsets of $\lz$, since it identifies the cases where the restriction of the $\lz$-topology is locally convex.

\smallskip

The structure of the paper is as follows. In Section \ref{sec: setup}, all the probabilistic and topological set-up is introduced. Section \ref{sec: UI} contains the first main result of the paper, Theorem \ref{thm: UI}, establishing a necessary and sufficient structural condition for a subset of $\lz$ to be uniformly integrable in an essentially measure-free way, as well as ramifications and discussion of Theorem \ref{thm: UI}. Section \ref{sec: local_conc} contains the second main result of the paper, Theorem \ref{thm: local conv}, connecting local convexity of subsets of $\lzp$ with their uniform $\qprob$-integrability under some $\qprob \sim \prob$; furthermore, discussion on the assumptions and conclusions of Theorem \ref{thm: local conv} is offered. Finally, Appendix \ref{sec: proof} contains the technical part of the proof of Theorem \ref{thm: UI}.

\section{Probabilistic Set-Up and Terminology}
\label{sec: setup}

Let $(\Omega, \F, \prob)$ be a probability space. For a probability $\qprob$ on $(\Omega, \F)$, we write $\qprob \ll \prob$ whenever $\qprob$ is absolutely continuous with respect to $\prob$ on $\F$; similarly, we write $\qprob \sim \prob$ whenever $\qprob$ and $\prob$ are equivalent on $\F$. All probabilities equivalent to $\prob$ have the same sets of zero measure, which shall be called \textsl{null}. Relationships between random variables are understood in the $\prob$-a.s. sense.

By $\lz$ we shall denote the set of all (equivalence classes modulo null sets of) real-valued random variables on $(\Omega, \F)$; furthermore, $\lzp$ will consist of all $X \in \lz$ such that $X \geq 0$. We follow the usual practice of not differentiating between a random variable and the equivalence class it generates in $\lz$. The expectation of $X \in \lzp$ under a probability $\qprob \ll \prob$ is denoted by $\expecq [X]$. We define a metric on $\lz$ via $\lz \times \lz \ni (X, Y) \mapsto \expecp \bra{ 1 \wedge \abs{X - Y} }$. The topology on $\lz$ that is induced by the previous metric depends on $\prob$ only through the null sets; convergence of sequences in this topology is simply convergence in measure under any probability $\qprob$ with $\qprob \sim \prob$. Unless otherwise explicitly stated, all topological concepts (convergence, closure, etc.) will be understood under the aforementioned metric topology on $\lz$.

A set $\X \subseteq \lz$ is called \textsl{convex} if $\pare{\alpha X + (1 -   \alpha) Z} \in \X$ whenever $X \in \X$, $Z \in \X$ and $\alpha \in [0,1]$. The set $\conv \X \subseteq \lz$ will denote the \textsl{convex hull} of $\X \in \lz$; namely, $\conv \X$ is the collection of all elements of the form $\sum_{i = 1}^k \alpha_i X_i$, where $k$ ranges in $\Natural$, $X_i \in \X$ and $\alpha_i \geq 0$ for all $i \in \set{1, \ldots, k}$, and $\sum_{i =1}^k \alpha_i = 1$. Furthermore, $\oX$ denotes the closure of $\X \in \lz$. The set $\X \subseteq \lz$ is called \textsl{bounded} if $\lim_{n \to \infty} \sup_{X \in \X} \prob[|X| > n] = 0$ holds; in this case, $\lim_{n \to \infty} \sup_{X \in \X} \qprob[|X| > n] = 0$ also holds for all probabilities $\qprob \ll \prob$. Note that boundedness in this sense coincides with boundedness in the sense of topological vector spaces---see \cite[Definition 5.36, page 186]{MR2378491}. For a probability $\qprob \ll \prob$, a set $\X \subseteq \lz$ is called \textsl{uniformly $\qprob$-integrable} if $\limn \pare{\sup_{X \in \X} \expecq \bra{ |X| \indic_{\set{|X| > n}} } }= 0$.

We now specialize to subsets of $\lzp$. The set $\X \subseteq \lzp$ is called \textsl{solid} if for all $Y \in \lzp$ and $X \in \X$ with $Y \leq X$, it follows that $Y \in \X$. The \textsl{solid hull} of $\X \subseteq \lzp$ is defined to be $\set{Y \in \lzp \such Y \leq X \text{ for some } X \in \X}$; clearly, it is the smallest solid subset of $\lzp$ that contains $\X$. The set  $\soco \X$ will denote the solid hull of the convex hull of $\X \subseteq \lzp$; in other words,
\[
\soco \X \dfn \set{Y \in \lzp \such Y \leq Z \text{ for some } Z \in \conv \X}.
\]
It is straightforward to check that the solid hull of a convex set is convex; therefore, $\soco \X$ is the smallest solid and convex set that includes $\X \subseteq \lzp$. In fact, the operations of taking the convex and solid hull of a subset of $\lzp$ commute. Indeed, the next result implies in particular that $\soco \X = \conv \set{Y \in \lzp \such Y \leq X \text{ for some } X \in \X}$; therefore, whenever $\X \subseteq \lzp$ is a solid set, $\soco \X = \conv \X$ holds.

\begin{prop} \label{prop: soco}
Let $\X \subseteq \lzp$ be a solid set. Then, $\conv \X$ is also solid.
\end{prop}

\begin{proof}
Define a non-decreasing sequence $(\C^k)_{\kin}$ of subsets of $\lzp$ as follows: $\C^0 = \X$ and, inductively, for $\kin$, $\C^k \dfn \set{\alpha X + (1 - \alpha) Y \such X \in \X, \, Y \in \C^{k-1}, \, \alpha \in [0,1]}$. Note that $\conv \X = \bigcup_{\kin} \C^k$. In order to prove the solidity of $\conv \X$, it suffices to establish the solidity of $\C^k$ for each $\kin$. The latter will follow via an induction argument. Start by noting that $\C^0 = \X$ is solid by assumption. Now, fix $\kin$ and suppose that $\C^{k-1}$ is solid; we shall then show that $\C^k$ is also solid. Let $Z \in \lzp$ be such that $Z \leq \alpha X+ (1 - \alpha) Y$ for some $X \in \X$, $Y \in \C^{k-1}$ and $\alpha \in [0,1]$. We claim that there exist $X' \in \X$ and $Y' \in \C^{k-1}$ such that $Z = \alpha X' + (1 - \alpha) Y'$. To wit, if $\alpha = 0$ let $X' = 0$ and $Y' = Z$, while if $\alpha = 1$ set $X' = Z$ and $Y' = 0$. Assume then that $\alpha \in (0,1)$. Note that $\set{X < Z} \cap \set{Y < Z} = \emptyset$. Define
\begin{eqnarray*}
X' &=& Z \indic_{\set{Z \leq X \wedge Y}} + \frac{Z - (1 - \alpha) Y}{\alpha} \indic_{\set{Y < Z \leq X} } + X \indic_{\set{X < Z \leq Y}}, \\
Y' &=& Z \indic_{\set{Z \leq X \wedge Y}} + \frac{Z - \alpha X}{1 - \alpha} \indic_{\set{X < Z \leq Y}} + Y \indic_{\set{Y < Z \leq X}}.
\end{eqnarray*}
Since $Z \leq \alpha X + (1 - \alpha) Y$, it is straightforward to check that $X' \in \lzp$, $Y' \in \lzp$ and $X' \leq X$, $Y' \leq Y$. Therefore, by the induction hypothesis and the solidity of $\X$, $X' \in \X$ and $Y' \in \C^{k-1}$. Furthermore, by definition of $X'$ and $Y'$, the equality $Z = \alpha X' + (1 - \alpha) Y'$ follows in a straightforward way, completing the proof.
\end{proof}

Whenever $\X \subseteq \lzp$ is convex and solid, the set $\oX$ is again convex and solid. (For the latter solidity property, let $Z \in \oX$ and $Y \in \lzp$ with $Y \leq Z$. Assume that the $\X$-valued sequence $(Z_n)_{\nin}$ is such that $\limn Z_n = Z$. Then, $(Z_n \wedge Y)_{\nin}$ is still $\X$-valued since $\X$ is solid, and $\limn (Z_n \wedge Y) = Z \wedge Y = Y$, which implies that $Y \in \oX$. Therefore, $\oX$ is solid as well.) It follows that $\oX$ is the smallest convex, solid and closed subset of $\lzp$ that contains the convex and solid $\X \subseteq \lzp$.

\section{A Structural Condition for the Essentially Measure-Free Version of Uniform Integrability of Sets in $\lz$} \label{sec: UI}

\subsection{The first main result}

We begin with a simple result giving an equivalent formulation of uniform integrability that will tie better with Theorem \ref{thm: UI} (which immediately follows). Recall that $\Real \ni x \mapsto x_+ \in \Real_+$ denotes the operation returning the positive part of a real number.

\begin{prop} \label{prop: UI_opt}
Let $\qprob \ll \prob$ and $\X \subseteq \lzp$. Then, $\X$ is uniformly $\qprob$-integrable  if and only if $\limn \pare{ \sup_{X \in \X} \expecq \bra{ (X - n)_+ } } = 0$ holds.
\end{prop}

\begin{proof}
Define $x \dfn \limn \pare{ \sup_{X \in \X} \expecq \bra{ (X - n)_+ } }$ and $y \dfn \limn \pare{ \sup_{X \in \X} \expecq \bra{ X \indic_{\set{X > n}} } }$. Since $(X - n)_+ \leq X \indic_{\set{X > n}}$ holds for all $X \in \lzp$ and $\nin$, it follows that $x \leq y$. We shall actually show that $x = y$, which will imply the result of Proposition \ref{prop: UI_opt}. If $y = \infty$, then $\sup_{X \in \X} \expecq \bra{ X } = \infty$, which implies that $\sup_{X \in \X} \expecq \bra{ (X - n)_+ } = \infty$ holds for all $\nin$, so that $x = \infty$. If $y \in \Real_+$, note that $y_0 \dfn \sup_{X \in \X} \expecq \bra{ X } \in \Real_+$. For each $\nin$, pick $X_n \in \X$ such that $\expecq \bra{ X_n \indic_{\set{X_n > n^2}} }  \geq y (n-1) / n$ holds. Since Markov's inequality implies that $n \prob \bra{X_n > n^2} \leq \expec \bra{X_n} / n \leq y_0 / n$ holds for all $\nin$, we estimate
\[
\expecq \bra{ (X_n - n)_+ } \geq \expecq \bra{ (X_n - n)_+ \indic_{\set{X_n > n^2} } } =  \expecq \bra{ X_n \indic_{\set{X_n > n^2}} } - n \prob \bra{X_n > n^2} \geq y - \frac{y + y_0}{n},
\]
for all $\nin$. It follows that $x \geq \limsup_{n \to \infty} \expecq \bra{ (X_n - n)_+ } \geq y$, which combined with $x \leq y$ gives $x = y$. The above discussion shows that $x = y$ always holds, which concludes the proof.
\end{proof}

In preparation for the first main result of the paper, for $\X \subseteq \lzp$ define
\begin{equation} \label{eq: X_opt}
\On \dfn \soco  \set{(X - n)_+ \such X \in \X}, \quad \text{for all } \nin.
\end{equation}
It is straightforward to check that $\pare{\On}_{\nin}$ is a non-increasing sequence of subsets of $\lzp$. Note that the sequence $\pare{ \set{(X - n)_+ \such X \in \X} }_{\nin}$ of subsets of $\lzp$ will not necessarily have the same property, except when $\X$ is solid; in this case, $\set{(X - n)_+ \such X \in \X} \subseteq \lzp$ is also solid, and $\On = \conv  \set{(X - n)_+ \such X \in \X}$ holds in view of Proposition \ref{prop: soco}.

\begin{thm} \label{thm: UI}
Let $\X \subseteq \lzp$ and define the non-increasing sequence $(\On)_{\nin}$ of subsets of $\lzp$ as in \eqref{eq: X_opt}. Then, the following statements are equivalent:
\begin{enumerate}
	\item $\bigcap_{\nin} \On = \set{0}$.
	\item There exists a probability $\qprob \sim \prob$ such that $\X$ is uniformly $\qprob$-integrable.
\end{enumerate}
\end{thm}

\begin{proof}
Implication $(1) \Rightarrow (2)$ is quite technical, and is discussed in Appendix \ref{sec: proof}. Implication $(2) \Rightarrow (1)$ is almost immediate; however, we shall prove a stronger conclusion, namely, that condition (2) implies $\bigcap_{\nin} \oOn = \set{0}$, which in turn trivially implies condition (1). The reason for establishing this (seemingly) stronger implication is that the latter will be used in the proof of Proposition \ref{prop: decomposition} later on.

Assume condition (2) of Theorem \ref{thm: UI}. As $(\oOn)_{\nin}$ is a non-increasing sequence of convex, solid and closed subsets of $\lzp$ (see discussion after the proof of Proposition \ref{prop: soco}), it follows in a straightforward way that $\bigcap_{\nin}\oOn$ coincides with the set of all the limits of sequences $(Z_n)_{\nin}$ with the property that $Z_n \in \On$ holds for all $\nin$. Given the existence of a probability $\qprob \sim \prob$ as in condition (2), any such sequence $(Z_n)_{\nin}$ with $Z_n \in \On$ for all $\nin$ satisfies
\[
\limsup_{n \to \infty} \expecq \bra{Z_n} \leq \limsup_{n \to \infty} \pare{\sup_{X \in \X} \expecq \bra{(X-n)_+}} = 0,
\]
as follows from Proposition \ref{prop: UI_opt} and the convexity of $\Real \ni x \mapsto x_+$. Since $Z_n \in \lzp$ for all $\nin$, it follows that $(Z_n)_{\nin}$ converges to zero for the $\Lb^1(\qprob)$-topology, which implies that $\limn Z_n = 0$. Therefore, $\bigcap_{\nin}\oOn = \set{0}$, which also implies that $\bigcap_{\nin} \On = \set{0}$.
\end{proof}

\begin{rem} \label{rem: general L_zero}
A set $\X \subseteq \lz$ is uniformly $\qprob$-integrable for some probability $\qprob \ll \prob$ if and only if $\set{|X| \such X \in \X} \subseteq \lzp$ is uniformly $\qprob$-integrable. It then follows that Theorem \ref{thm: UI} can be extended to the case where $\X$ is an arbitrary subset of $\lz$, provided that one alters the definition of $\On$ in \eqref{eq: X_opt} with $\On \dfn \soco \set{(|X| - n)_+ \such X \in \X}$ for all $\nin$.
\end{rem}

\subsection{Connections of Theorem \ref{thm: UI} with Financial Mathematics} \label{subsec: finance}

Suppose that $\X \subseteq \lzp$ represents financial positions available at some future time $T$. For $X \in \X$ and $k \in \Real_+$, the random variable $(X - k)_+$ is the payoff of an option to receive the position $X$ upon paying $k$ units of cash at time $T$. Given this interpretation, the set $\set{(X - n)_+ \such X \in \X}$ coincides with all options to buy $X \in \X$ for the fixed strike price $n \in \Natural$. A probability $\qprob \sim \prob$ can be used for valuation of financial contracts, assigning the value $\expecq[Z]$ to a contract that will pay the amount $Z \in \lzp$ at time $T$. Given the previous understanding, the statement of Theorem \ref{thm: UI} has the following financial interpretation: there either exists a valuation probability $\qprob \sim \prob$ such that the value of options of the form $(X-n)_+$ for $X \in \X$ under $\qprob$ converges to zero as $n$ tends to infinity \emph{uniformly} over all $X \in \X$, or the structure of $\X$ is rich enough to allow for the possibility of super-hedging a fixed positive (non-zero) position using convex combinations of options with arbitrary large strike prices, in the sense that there exists $Y \in \lzp$ with $\prob \bra{Y > 0} > 0$ and a sequence $(Z_n)_{\nin}$ such that $Y \leq Z_n \in \conv \set{(X - n)_+ \such X \in \X}$ holds for all $\nin$.

The discussion of the previous paragraph applies also to options allowing exchange of positions in $\X$ for units of some random payoff other than cash. This becomes easier understood via use of the method of \num-change. (For an illustration of this technique in a dynamic semimartingale environment, see \cite{MR1381678}.) In accordance to \eqref{eq: X_opt}, for $Y \in \lzp$ define $\Oy \dfn \soco \set{(X - Y)_+ \such X \in \X}$. For $Y \in \lzp$ with $\prob \bra{Y > 0} = 1$, the set $Y^{-1} \X = \set{ X/ Y \such X \in \X}$ consists of positions in $\X$ denominated in units of $Y$; in other words, $Y$ is used as a \num. Since $(X/Y - n)_+ = Y^{-1} (X - n Y)_+$ holds for all $X \in \X$ and $\nin$, it is straightforward to check that $\mathcal{O}_n^{Y^{-1} \X} = Y^{-1} \Ony$ holds for all $\nin$. Note that statement (2) of Theorem \ref{thm: UI} is invariant under such changes of \num: if it holds for $\X$, it also holds for $Y^{-1} \X$ whenever $Y \in \lzp$ is such that $\prob \bra{Y > 0} = 1$; it then follows that
\[
\bigcap_{\nin} \On = \set{0} \Longleftrightarrow \bigcap_{\nin} \mathcal{O}_n^{Y^{-1} \X} = \set{0} \Longleftrightarrow \bigcap_{\nin} Y^{-1} \Ony = \set{0} \Longleftrightarrow \bigcap_{\nin} \Ony = \set{0}.
\]

\subsection{A decomposition result} \label{subsec: decomp}

The next is an interesting ``decomposition'' result, which is mainly a corollary of Theorem \ref{thm: UI} (in the form discussed in Remark \ref{rem: general L_zero}).

\begin{prop} \label{prop: decomposition}
A set $\X \subseteq \lz$ either fails to be uniformly $\qprob$-integrable for all probabilities $\qprob \ll \prob$, or there exists some probability $\qprob_\X \ll \prob$ with the following properties:
\begin{itemize}
	\item $\X$ is uniformly $\qprob_\X$-integrable, and
	\item whenever a probability $\qprob \ll \prob$ is such that $\qprob \bot \qprob_\X$, $\X$ fails to be uniformly $\qprob$-integrable.
\end{itemize}
\end{prop}

\begin{proof}
By considering $\set{|X| \such X \in \X} \subseteq \lzp$ instead of $\X$, we may assume that $\X \subseteq \lzp$. In the notation of \eqref{eq: X_opt}, define $\Ko \dfn \bigcap_{\nin} \oOn$. Note that $\Ko$ is convex, solid and closed, as it is the intersection of sets with the corresponding properties. It then follows in a straightforward way that there exists $\Omega_w \in \F$ such that:
\begin{itemize}
	\item $\prob \bra{\Omega_w \cap \set{Y > 0} } = 0$ holds for all $Y \in \Ko$.
	\item for any $A \subseteq \Omega \setminus \Omega_w$ with $\prob \bra{A} > 0$, there exists $Z \equiv Z_A \in \Ko$ such that $\prob \bra{A \cap \set{Z > 0}} > 0$ holds.
\end{itemize}
(Clearly, such $\Omega_w \in \F$ is unique modulo null sets.) If $\prob \bra{\Omega_w} = 0$, Theorem \ref{thm: UI} implies that $\X$ fails to be uniformly $\qprob$-integrable for all probabilities $\qprob \ll \prob$. On the other hand, if $\prob \bra{\Omega_w} > 0$, then using the notation $A^\qprob \dfn \set{\ud \qprob / \ud \prob > 0}$ for probabilities $\qprob \ll \prob$, Theorem \ref{thm: UI} implies that there exists $\qprob_\X \ll \prob$ with $A^{\qprob_\X} = \Omega_w$ (modulo null sets) such that $\X$ is uniformly $\qprob_\X$-integrable. In this case, when $\qprob \ll \prob$ is such that $\qprob \bot \qprob_\X$ then $A^{\qprob} \subseteq \Omega \setminus \Omega_w$, which implies again by Theorem \ref{thm: UI} that $\X$ fails to be uniformly $\qprob$-integrable. 
\end{proof}

\begin{rem}
Let $\X \subseteq \lz$, and suppose that $\X$ is uniformly $\qprob$-integrable for some probability $\qprob \ll \prob$. In this case, if both $\qprob_\X \ll \prob$ and $\qprob'_\X \ll \prob$ have the properties mentioned in Proposition \ref{prop: decomposition}, it necessarily holds that $\qprob_\X \sim \qprob'_\X$.
\end{rem}

\begin{rem} \label{rem: Bra-Sch}
In \cite{MR1768009}, given a convex set $\X \subseteq \lzp$, the authors show that there exists a set $\Omega_b \in \F$ such that $\indic_{\Omega_b} \X$ is bounded while $\X$ is \emph{hereditarily} unbounded on $\Omega \setminus \Omega_b$ in the sense that $\indic_A \X$ fails to be bounded for all $A \in \F$ with $A \subseteq \Omega \setminus \Omega_b$ and $\prob \bra{A} > 0$. The set $\Omega_b$ satisfying the previous property is necessarily unique (modulo null-sets). Proposition \ref{prop: decomposition} can be seen as a result in this direction; indeed, with the notation in its proof, given $\X \subseteq \lz$, it is shown that there exists a set $\Omega_w \in \F$ such that $\indic_{\Omega_w} \X$ is ``weakly compactizable'' (in the sense that there exists $\qprob \sim \prob$ such that $\indic_{\Omega_w} \X$ is uniformly $\qprob$-integrable) while $\indic_{\Omega \setminus \Omega_w} \X$ ``hereditarily fails to be weakly-compactizable'' (in the sense that $\indic_A \X$ fails to be weakly-compactizable for all $A \in \F$ with $A \subseteq \Omega \setminus \Omega_w$ and $\prob \bra{A} > 0$).
\end{rem}

\section{Local Convexity}

\label{sec: local_conc}

\subsection{The second main result}

We start with a definition of a concept has played a major role in the theory of Financial Mathematics, usually utilized in an indirect manner---see, for example, \cite[Lemma A1.1]{MR1304434}.

\begin{defn}
Let $(X_n)_{\nin}$ be a sequence in $\lz$. Any $\lz$-valued sequence $(Y_n)_{\nin}$ with the property that $Y_n \in \conv \set{X_k \such n \leq k \in \Natural }$ for all $\nin$ will be called a \textsl{sequence of forward convex combinations of $(X_n)_{\nin}$}.
\end{defn}

Let us agree to call a convex set $\X \subseteq \lz$ \textsl{locally convex for the $\lz$-topology} if any element of $\X$ has a neighbourhood base (for the relative $\lz$-topology on $\X$) consisting of convex sets. (Such definition is classical in the case where $\X$ is a topological vector space; however, we only require $\X \subseteq \lz$ to be convex.) Suppose that a convex set $\X \subseteq \lz$ is locally convex for the $\lz$-topology; then, whenever $(X_n)_{\nin}$ is an $\X$-valued sequence  that converges to $X \in \X$, all sequences of forward convex combinations of $(X_n)_{\nin}$ also converge to $X$.

The second main result of the paper that follows connects, amongst other things, local convexity of the $\lz$-topology of $\X \subseteq \lzp$ with uniform $\qprob$-integrability of $\X$ for some $\qprob \sim \prob$, in the case where $\X$ is convex, solid and bounded.

\begin{thm} \label{thm: local conv}
Let $\X \subseteq \lzp$ be a convex, solid and bounded set. Then, the following statements are equivalent:
\begin{enumerate}
	\item Whenever $(X_n)_{\nin}$ is an $\X$-valued sequence that converges to zero, all sequences of forward convex combinations of $(X_n)_{\nin}$ also converge to zero.
	\item $0 \in \X$ has a neighbourhood base (for the relative $\lz$-topology) consisting of convex sets.
	\item Any $X \in \X$ has a neighbourhood base (for the relative $\lz$-topology) consisting of convex sets.
	\item The $\lz$-topology on $\X$ coincides with the $\Lb^1(\qprob)$-topology on $\X$ for some $\qprob \sim \prob$.
 	\item $\X$ is uniformly $\qprob$-integrable with respect to some $\qprob \sim \prob$.
\end{enumerate}
\end{thm}

\begin{proof}
The chain of implications $(5) \Rightarrow (4) \Rightarrow (3) \Rightarrow (2) \Rightarrow (1)$ in Theorem \ref{thm: local conv} is straightforward.

Assume that condition $(5)$ fails. Define $\X'_n \dfn \set{(X - n)_+ \such X \in \X}$ for all $\nin$. Then, $\X'_n$ is a solid subset of $\X$, in view of the solidity of $\X$; furthermore, $\On = \conv \X'_n$ holds for all $\nin$ according to \eqref{eq: X_opt} and Proposition \ref{prop: soco} (see also the discussion before the statement of Theorem \ref{thm: UI}). In view of Theorem \ref{thm: UI}, it holds that $\bigcap_{\nin} \On \supsetneq \set{0}$; therefore, there exists $Y \in \lzp$ with $\prob \bra{Y > 0} > 0$ such that $Y \in \conv \X'_n$ holds for each $\nin$. Since $Y$ is a convex combination of elements in $\X'_n$ for all $\nin$, one can construct a $\X$-valued sequence $(X_n)_{\nin}$ with the properties that $X_n \in \X'_n$ for all $\nin$ and the constant sequence $(Y_n)_{\nin}$ with $Y_n = Y$ for all $\nin$ is a sequence of forward convex combinations of $(X_n)_{\nin}$. Since $\X$ is assumed to be bounded,
\[
\limsup_{n \to \infty} \prob \bra{X_n > 0} \leq \ \downarrow \limn \pare{ \sup_{X \in \X} \prob \bra{X > n} } = 0
\]
holds, which implies that $\limn X_n = 0$. Therefore, we have constructed an $\X$-valued sequence $(X_n)_{\nin}$ that converges to $0 \in \X$, of which the (constant, and equal to $Y$) sequence $(Y_n)_{\nin}$ of its forward convex combinations fails to be convergent to zero (since $Y \in \lzp$ is such that with $\prob \bra{Y > 0} > 0$). This implies that condition $(1)$ also fails. Therefore, implication $(1) \Rightarrow (5)$ has been established as well.
\end{proof}

\subsection{Remarks on Theorem \ref{thm: local conv}}

We proceed with several remarks on the hypotheses and statement of Theorem \ref{thm: local conv}

\subsubsection{The structural condition} As mentioned in the proof of Theorem \ref{thm: local conv} (and comes as a consequence of Proposition \ref{prop: soco}), when $\X \subseteq \lzp$ is solid, the set $\On$ defined in  \eqref{eq: X_opt} is equal to $\conv \set{(X - n)_+ \such X \in \X}$ for all $\nin$. Therefore, in view of Theorem \ref{thm: UI}, the conditions of Theorem \ref{thm: local conv} are further equivalent to
\begin{equation} \label{eq: structural}
\bigcap_{\nin} \conv \set{(X - n)_+ \such X \in \X} = \set{0}.
\end{equation}

\subsubsection{The case of subsets of $\lz$} Theorem \ref{thm: local conv} can be extended to cover the case of $\X \subseteq \lz$ that is convex, bounded and solid, where the last property means that whenever $X \in \X$ and $Y \in \lz$ are such that $|Y| \leq |X|$, then $Y \in \X$. The details are straightforward (in this respect, see also Remark \ref{rem: general L_zero}) and are, therefore, omitted. Note that is $X \subseteq \lz$ is solid, $X \in \X$ implies $|X| \in \X$; from this, it is straightforward to see that $\set{(|X| - n)_+ \such X \in \X} = \set{(X - n)_+ \such X \in \X}$; therefore, the structural condition \eqref{eq: structural} remains exactly the same in this case.

\subsubsection{Local convexity at zero}

Let $\X \subseteq \lzp$ be a convex, solid and bounded set. As a consequence of Theorem \ref{thm: local conv}, local convexity of $\X$ for the $\lz$-topology is equivalent to local convexity of $\X$ for the $\lz$-topology only at $0 \in \X$. Clearly, solidity of $\X$ is crucial for this to be true.

\subsubsection{Local convexity and closure} \label{sss: local_conv_close}

Let $\X \subseteq \lzp$ be a convex, solid and bounded set. As Theorem \ref{thm: local conv} suggests, local convexity of $\X$ for the $\lz$-topology implies also local convexity of $\oX$ for the $\lz$-topology. Again, solidity of $\X$ is crucial for this to hold---see \S \ref{subsubsec: solid} later on.

\subsubsection{On boundedness}

Boundedness of a convex and solid set $\X \subseteq \lzp$ in the statement of Theorem \ref{thm: local conv} is clearly necessary in order to have uniform $\qprob$-integrability for some probability $\qprob \sim \prob$. When a convex and solid set $\X \subseteq \lzp$ fails to be bounded, local convexity of $\X$ for the $\lz$-topology in general may not even imply that the $\lz$-topology on $\X$ is the same as the $\Lb^1(\qprob)$-topology for some probability $\qprob \sim \prob$. Indeed, consider $\Omega = \Natural$, $\F$ the collection of all subsets of $\Omega$ and the probability $\prob$ satisfying $\prob \bra{ \set{i}} = 2^{-i}$ for all $i \in \Natural$. It is then straightforward to check that $\lz$ is isomorphic to $\Real^\Natural$ equipped with the product topology. In particular, $\X \equiv \lzp$ is locally convex for the $\lz$-topology. However, $\lzp$ even fails to be a subset of $\Lb^1(\qprob)$ for any $\qprob 
\sim \prob$.

More generally, Proposition \ref{prop: UI_unbdd} that follows will complement Theorem \ref{thm: local conv}. Note that, as a consequence of the bipolar theorem in $\lz$ \cite[Theorem 1.3]{MR1768009}, whenever $\X \subseteq \lzp$ is convex and solid, there exists $\Omega_b \in \F$ such that $\indic_{\Omega_b} \X$ is bounded while $\indic_{\Omega \setminus \Omega_b} \oX = \indic_{\Omega \setminus \Omega_b} \lzp$. Since the case of bounded $\X \subseteq \lzp$ is covered by Theorem \ref{thm: local conv}, we turn attention at what is happening in the hereditarily unbounded part (the notion of hereditary unboundedness has been introduced in \cite{MR1768009} and is reviewed in Remark \ref{rem: Bra-Sch}); it is then sufficient to focus on the case $\X = \lzp$. 

\begin{prop} \label{prop: UI_unbdd}
The set $\lzp$ is locally convex for the $\Lb^0$-topology if and only if the underlying probability space is purely atomic. In this case, the $\lz$-topology on $\lzp$ coincides with the $\Lb^1(\qprob)$-topology on $\X$ for some $\qprob \sim \prob$ if and only if the underlying probability space (is purely atomic and) has only a finite number of atoms (and any $\qprob \sim \prob$ will do in this csae).
\end{prop}

\begin{proof}
When the underlying probability space is purely atomic, $\lzp$ will be topologically isomorphic to either $\Real^n_+$ for some $\nin$ or to $\Real^\Natural_+$ (the latter spaces equipped with the usual product topology). In any case, the $\lz$-topology is locally convex. When the probability space fails to be purely atomic, one can find $A \in \F$ with $\prob \bra{A} > 0$ such that $A$ contains no atoms. Then, for each $m \in \Natural$ one can find a partition $(A_{m, k})_{k \in \set{1, \ldots, m}}$ of $A$ such that $\prob \bra{A_{m, k}} = \prob \bra{A} / m$ holds for all $k \in \set{1, \ldots, m}$. Defining the sequence $(X_n)_{\nin}$ via $X_n = m \indic_{A_{m, k}}$ whenever $n = 2^{m-1} + (k-1)$ for $m \in \Natural$ and $k \in \set{1, \ldots, m}$, it holds that $\limn X_n = 0$. However, the constant non-zero sequence $(Z_n)_{\nin}$ defined via $Z_n = \indic_A$ for all $\nin$ is a sequence of forward convex combinations of $(X_n)_{\nin}$. It follows that $\lzp$ cannot be locally convex for the $\lz$-topology.

When the underlying probability space is purely atomic with a finite number of atoms, all $\Lb^1(\qprob)$ spaces, as well as $\lz$, are topologically isomorphic to $\Real^n$, where $\nin$ is the number of distinct atoms. On the other hand, if the underlying probability space is purely atomic with a countably infinite number of atoms, it is straightforward to check that $\Lb_+^1(\qprob) \dfn \Lb^1(\qprob) \cap \lzp$ is a \emph{strict} subset of $\lzp$ for any $\qprob \sim \prob$. 
\end{proof}

\subsubsection{On solidity} \label{subsubsec: solid}

Define $\X \dfn \set{X \in \lzp \such \expecp \bra{X} = 1}$; clearly, $\X$ is convex and bounded. The solid hull of $\X$ is $\So = \set{X \in \lzp \such \expecp \bra{X} \leq 1}$, which is also convex and bounded. It is a consequence of \cite[Proposition 4.12]{MR1876169} that the $\lz$ and $\Lb^1(\prob)$ topologies on $\X$ coincide; in particular, $\X$ is locally convex for the $\lz$-topology. However, when the underlying probability space is non-atomic:

\begin{itemize}
	\item $\X$ fails to be uniformly $\qprob$-integrable for all probabilities $\qprob \sim \prob$. (Indeed, for each $\qprob \sim \prob$ it is straightforward to construct an $\X$-valued sequence $(X_n)_{\nin}$ such that $\limn X_n = 0$ holds, but $\liminf_{n \to \infty} \expecq \bra{X_n} > 0$.) Therefore, the equivalence of statements (4) and (5) in Theorem \ref{thm: local conv} may fail when $\X$ is not solid.
	\item If the underlying probability space is non-atomic, it holds that $\oX = \So$. (Indeed, $\oX \subseteq \So$ holds in view of Fatou's lemma. Conversely, since the underlying probability space is non-atomic, there exists an $\X$-valued sequence $(Y_n)_{\nin}$ such that $\limn Y_n = 0$ holds; then, for any $X \in \So$, the $\X$-valued sequence $(X_n)_{\nin}$ defined via $X_n = X + (1 - \expecp[X]) Y_n$ for all $\nin$ is such that $\limn X_n = X$.) Since $\So$ fails to be locally convex (which can be seen by a similar argument as in the proof of Proposition \ref{prop: UI_unbdd}), the closure of a locally convex set for the $\lz$-topology may fail to be locally convex for the $\lz$-topology. This is in direct contrast with the discussion in \S \ref{sss: local_conv_close} in the case where $\X$ is solid.
\end{itemize}

It is an open question whether the equivalence of statements (3) and (4) of Theorem \ref{thm: local conv} is valid under the assumption that $\X$ is convex and bounded. There does not seem to be a straightforward adaptation of the method of proof provided in this paper to cover this case. A related open question is whether the equivalence between statements (3), (4) and (5) of Theorem \ref{thm: local conv} is valid under the assumption that $\X$ is convex, bounded \emph{and closed}; note that the set $\X$ is the example above is not $\lz$-closed when the underlying probability space is non-atomic, although it is always $\Lb^1(\prob)$-closed.

\appendix

\section{Finishing the Proof of Theorem \ref{thm: UI}} \label{sec: proof}

In order to conclude the proof of Theorem \ref{thm: UI}, it remains to establish implication $(1) \Rightarrow (2)$. Before that, certain prerequisites on polar and bipolar sets will be discussed.

\subsection{Measures and polar sets} \label{subsec: meas_polars}

By $\Mp$ we shall denote the convex cone of all $\sigma$-finite (nonnegative) measures on $(\Omega, \F)$ that are absolutely continuous with respect to $\prob$. Convexity for subsets of $\Mp$ is defined in the usual sense. A set $\D \subseteq \Mp$ is solid if the conditions $\nu \in \D$ and $\mu \in \Mp$ with $\mu \leq \nu$ (in the sense that $\mu[A] \leq \nu[A]$ for all $A \in \F$) imply that $\mu \in \D$.

The set $\Mp$ can be placed in one-to-one correspondence with $\lzp$ via identifying $\mu \in \Mp$ with $(\ud \mu / \ud \prob) \in \lzp$. (Then, convexity and solidity in $\Mp$ are naturally identified with the corresponding properties in $\lzp$.) We endow $\Mp$ with the natural ($\lz$-)topology that comes through this identification. Note that this topology depends on the representative probability $\prob$ only through its null sets. In what follows, topological notions (such as closure) for subsets of $\Mp$ are always understood in this sense.

The next auxiliary result shows that a solid subset of $\Mp$ whose closure is the whole orthant $\Mp$ contains ``rays'' through measures that are ``arbitrarily close to being equivalent to $\prob$.''

\begin{lem} \label{lem: unbdd}
Let $\D \subseteq \Mp$ be solid, and suppose that $\oD = \Mp$. Then, for all $\epsilon \in (0,1)$ there exists $\mu \equiv \mu_\epsilon \in \Mp$ with $\prob \bra{\ud \mu / \ud \prob = 0} < \epsilon$ such that $a \mu \in \D$ holds for all $a \in \Real_+$.
\end{lem}

\begin{proof}
Fix $\epsilon \in (0,1)$. Since $k \prob \in \Mp = \oD$ for each $\kin$ and $\D$ is solid, it follows that there exists $A_k \in \F$ with the properties that $\prob \bra{A_k} > 1 - \epsilon / 2^k$ and $k \mu_k \in \D$, where $\ud \mu_k / \ud \prob = \indic_{A_k}$. Define $A \dfn \bigcap_{\kin} A_k \in \F$ and $\mu$ via $\ud \mu \dfn \indic_A \ud \prob$. Note that $\prob[\ud \mu / \ud \prob = 0] = \prob[\Omega \setminus A] < \epsilon$. Furthermore, $k \mu \leq k \mu_k$ holds for all $\kin$; since $k \mu_k \in \D$ and $\D$ is solid, $k \mu \in \D$ holds for all $\kin$. The solidity of $\D$ gives that $a \mu \in \D$ for all $a \in \Real_+$.
\end{proof}

For two sets $\C \subseteq \lzp$ and $\D \subseteq \Mp$, their polars $\C^\circ \subseteq \Mp$ and $\D^\circ \subseteq \lzp$ are defined via
\[
\C^\circ \dfn \set{\mu \in \Mp \such \sup_{X \in \C} \inner{\mu}{X} \leq 1}, \quad \D^\circ \dfn \set{X \in \lzp \such \sup_{\mu \in \D} \inner{\mu}{X} \leq 1}.
\]
It is straightforward to check that polar sets are convex and solid; furthermore, Fatou's lemma implies that polar sets are closed. The polar of $\set{0} \subseteq \lzp$ is $\Mp$ and the polar of $\set{0} \subseteq \Mp$ is $\lzp$.  Furthermore, it is straightforward to check by the definition of polarity that $\pare{\bigcup_{\nin} \C_n}^\circ = \bigcap_{\nin} \C_n^\circ$ holds for any sequence $(\C_n)_{\nin}$ of subsets of $\lzp$; similarly, $\pare{\bigcup_{\nin} \D_n}^\circ = \bigcap_{\nin} \D_n^\circ$ holds for any sequence $(\D_n)_{\nin}$ of subsets of $\Mp$.

Define also the \textsl{bipolar} of $\C \subseteq \lzp$ or $\D \subseteq \Mp$ via $\C^{\circ \circ} \dfn (\C^\circ)^\circ$ and $\D^{\circ \circ} \dfn (\D^\circ)^\circ$. It is straightforward to check that $\C \subseteq \C^{\circ \circ}$ and $\D \subseteq \D^{\circ \circ}$. Bipolar sets in either $\lzp$ or $\Mp$ are convex, solid and closed. In fact, the version of the bipolar theorem in $\lz$ \cite[Theorem 1.3]{MR1768009} implies that $\C^{\circ \circ}$ is the smallest convex, solid and closed subset of $\lzp$ which contains $\C \subseteq \lzp$; similarly, $\D^{\circ \circ}$ is the smallest convex, solid and closed subset of $\Mp$ which contains $\D \subseteq \Mp$. In particular, if $\C \subseteq \lzp$ (or $\D \subseteq \Mp$) is already convex and solid, then since $\oC$ (or $\oD$) is also convex and solid (see discussion in the end of Section \ref{sec: setup}), it follows that $\C^{\circ \circ} = \oC$ (or $\D^{\circ \circ} = \oD$).

\subsection{Proof of implication $(1) \Rightarrow (2)$ of Theorem \ref{thm: UI}}

In order to ease the reading, the sets $\On$ of \eqref{eq: X_opt} will be denoted here by $\K_n$ for all $\nin$; similarly, we shall write $\oK_n$ in place of $\oOn$, for all $\nin$.
We shall assume in the sequel that $\bigcap_{\nin} \K_n = \set{0}$. Recall that $\K_n$ is a solid subset of $\lzp$ for all $\nin$. It is then a consequence of the next result that the condition $\bigcap_{\nin} \K_n = \set{0}$ actually implies the seemingly stronger $\bigcap_{\nin} \oK_n = \set{0}$.

\begin{lem} \label{lem: solid_closure_zero}
Let $(\So_n)_{\nin}$ be a sequence of solid subsets of $\lzp$. Then, $\bigcap_{\nin} \So_n = \emptyset$ holds if and only if $\bigcap_{\nin} \oS_n = \emptyset$.
\end{lem}

\begin{proof}
Of course, $\bigcap_{\nin} \oS_n = \emptyset$ always implies $\bigcap_{\nin} \So_n = \emptyset$. Conversely, assume that $\bigcap_{\nin} \oS_n \supsetneq \emptyset$; we shall then establish that $\bigcap_{\nin} \So_n \supsetneq \emptyset$. Let $Z \in \bigcap_{\nin} \oS_n$ be such that $\prob \bra{Z > 0} > 0$. For all $\nin$, since $Z \in \oS_n$, let $Z_n \in \So_n$ be such that $\prob \bra{|Z_n - Z| >  Z /2 \such Z > 0} < 2^{- \pare{n + 1}}$. Let $A \dfn \bigcap_{\nin} \set{|Z_n - Z| \leq  Z /2}$. Since $\prob \bra{Z > 0} > 0$ and
\[
\prob[A \such Z > 0] \geq 1 - \sum_{\nin} \prob[|Z_n - Z| >  Z /2 \such Z > 0] \geq 1 - \frac{1}{2} = \frac{1}{2} > 0,
\]
it follows that $\prob \bra{A \cap \set{Z > 0}} > 0$; therefore, upon defining $Y \dfn (Z/2) \indic_A \in \lzp$, note that $\prob \bra{Y > 0} > 0$. Furthermore, $Y \leq Z_n$ follows from the fact that $|Z_n - Z| \leq  Z /2$ holds on $A$ for all $\nin$, which implies that $Y \in \So_n$ in view of the solidity of $\So_n$ for all $\nin$. It follows that $Y \in \bigcap_{\nin} \So_n$, which completes the proof.
\end{proof}

For all $\nin$, associated with $\K_n$ define a function $u_n : \Mp \mapsto [0, \infty]$ via
\begin{equation} \label{eq: kappa_n}
u_n (\mu) \dfn \sup_{X \in \K_n} \inner{\mu}{X} = \sup_{X \in \oK_n} \inner{\mu}{X}, \quad \mu \in \Mp,
\end{equation}
where the second equality follows from Fatou's lemma. Each $u_n$, $\nin$, has the following properties:
\begin{itemize}
	\item \emph{monotonicity}: $u_n(\mu) \leq u_n(\nu)$ holds for all $\mu \in \Mp$ and $\nu \in \Mp$ with $\mu \leq \nu$.
	\item \emph{sub-additivity}: $u_n(\mu + \nu) \leq u_n(\mu) + u_n(\nu)$ holds for all $\mu \in \Mp$ and $\nu \in \Mp$.
	\item \emph{positive homogeneity}: $u_n (a \mu) = a u_n(\mu)$ holds for all $\mu \in \Mp$ and $a \in (0, \infty)$.
	\item \emph{continuity from below}: for any non-decreasing $\Mp$-valued sequence $(\mu_k)_{\kin}$ such that $\mu \dfn \limk \mu_k = \bigvee_{\kin} \mu_k$ is an element of $\Mp$, it holds that $u_n (\mu) = \limk u_n(\mu_k)$. 
\end{itemize}
The first three properties are immediate from the definition of $u_n$, $\nin$, while the last property follows from the monotone convergence theorem.

The following intermediate result will be helpful.

\begin{lem} \label{lem: bddness}
With the above notation, $\bigcap_{\nin} \K_n = \set{0}$ implies that there exists $\mu_0 \in \Mp$ with $\mu_0 \sim \prob$ such that $u_n(\mu_0) \leq 1$ holds for all $\nin$.
\end{lem}

\begin{proof}
By Lemma \ref{lem: solid_closure_zero}, $\bigcap_{\nin} \K_n = \set{0}$ implies $\bigcap_{\nin} \oK_n = \set{0}$. Note that each $\oK_n$ is convex, solid and closed. According to \cite[Lemma 2.3]{MR1768009}, there exist sets $B_n \in \F$ for all $\nin$ with the property that $\indic_{B_n} \oK_n$ is bounded while $\indic_{\Omega \setminus B_n} \oK_n = \indic_{\Omega \setminus B_n} \lzp$. By definition of the sets $\oK_n$, $\nin$, it is straightforward to check that for any $\nin$ and $Z_{n+1} \in \oK_{n+1}$ there exists $Z_n \in \oK_n$ such that $Z_n \geq (Z_{n+1} - 1)_+$; therefore, it follows that $B_n = B_1$ holds for all $\nin$ (up to null sets). This fact implies that $\bigcap_{\nin} \oK_n \supseteq \indic_{\Omega \setminus B_1} \lzp$; therefore, in view of $\bigcap_{\nin} \oK_n = \set{0}$, $\prob \bra{B_1} = 1$ has to hold. It follows that $\oO_1$ is bounded. An application of \cite[Theorem 1.1(4)]{Kar10a} implies that there exists $Z \in \oO_1$ such that $\sup_{Y \in \oO_1} \expecp \bra{Y / (1 + Z)} \leq 1$. Defining $\mu_0 \sim \prob$  via the recipe $\ud \mu_0 / \ud \prob = 1 / (1 + Z)$, the result follows.
\end{proof}

Along with the non-increasing sequence of mappings $(u_n)_{\nin}$ that was introduced above, define also the ``limiting'' map $u : \Mp \mapsto [0, \infty]$ via
\[
u (\mu) \dfn \, \downarrow \limn u_n(\mu), \quad \mu \in \Mp.
\]
The monotonicity, sub-additivity and positive homogeneity properties of the sequence $(u_n)_{\nin}$ directly transfer to $u$. (Note, however, that continuity from below will fail in general for $u$.)

\begin{rem} \label{rem: unif_int}
By Proposition \ref{prop: UI_opt}, it follows that $u (\qprob) = 0$ for some probability $\qprob \in \Mp$ is equivalent to uniform $\qprob$-integrability of $\X$.
\end{rem}

For all $\nin$, define $\eL_n \dfn \set{\mu \in \Mp \such u_n(\mu) \leq 1}$;
recalling the discussion of polar sets in Subsection \ref{subsec: meas_polars}, note that $\eL_n = \K^\circ_n = \oK^\circ_n$. In particular, $\eL_n$ is a convex, solid and closed subset of $\Mp$ for all $\nin$. Furthermore, from the bipolar theorem in $\lz$ \cite[Theorem 1.3]{MR1768009} and the fact that $\oK_n$ is convex, solid and closed (see discussion after the proof of Proposition \ref{prop: soco} at the end of Section \ref{sec: setup}), it follows that $\eL_n^\circ = \pare{\oK_n}^{\circ \circ} = \oK_n$ for all $\nin$. 

Let $\eL \dfn \bigcup_{\nin} \eL_n$. Note that $\eL$ is a solid and convex subset of $\Mp$. Since $\bigcap_{\nin} \oK_n = \set{0}$, it follows that
\[
\oL = \eL^{\circ \circ} = \pare{\bigcup_{\nin} \eL_n}^{\circ \circ} = \pare{\pare{\bigcup_{\nin} \eL_n}^\circ}^{\circ} = \pare{ \bigcap_{\nin} \eL_n^\circ}^\circ = \pare{\bigcap_{\nin} \oK_n}^\circ = \set{0}^\circ = \Mp.
\]
By Lemma \ref{lem: unbdd}, for each $\kin$ there exists $\mu_k \in \Mp$ with $\prob \bra{\ud{\mu_k} / \ud \prob = 0} < 1/k$ such that $m \mu_k \in \bigcup_{\nin} \eL_n$ holds for all $m \in \Natural$. Fixing $\kin$, for each $m \in \Natural$ pick $n_m = n_m(k) \in \Natural$ such that $m \mu_k \in \eL_{n_m}$; then, $u (\mu_k) \leq u_{n_m} (\mu_k) \leq 1 / m$ would hold for all $\kin$, which implies that $u(\mu_k) = 0$. Recall from Lemma \ref{lem: bddness} that there exists $\mu_0 \sim \prob$ such that $u_n(\mu_0) \leq 1$ holds for all $\nin$. Define a new $\Mp$-valued sequence $(\nu_k)_{\kin}$ such that $\ud \nu_k / \ud \prob = \pare{\ud \mu_k / \ud \prob} \wedge \pare{\ud \mu_0 / \ud \prob}$ holds for all $\kin$; then, it follows in a straightforward way that $\prob \bra{\ud{\nu_k} / \ud \prob = 0} = \prob \bra{\ud{\mu_k} / \ud \prob = 0} < 1/k$, $u(\nu_k) \leq u(\mu_k) = 0$ and $\nu_k \leq \mu_0$ for all $\kin$, where $u_1(\mu_0) \leq 1$.

Define $\nu \dfn \sum_{\kin} 2^{-k} \nu_k$; note that $\nu \leq \mu_0$, so that $\nu \in \Mp$. Since $\prob \bra{\ud{\nu_k} / \ud \prob = 0} < 1/k$ holds for all $\kin$, it is straightforward that $\nu \sim \prob$. Furthermore, $\nu \leq \sum_{k=1}^{m} 2^{-k} \nu_k + 2^{-m} \mu_0$ holds for all $m \in \Natural$; therefore,
\[
u(\nu) \leq u \pare{\sum_{k=1}^{m} 2^{-k} \nu_k} + u \pare{2^{-m} \mu_0}  \leq \sum_{k=1}^{m} 2^{-k} u(\nu_k) + 2^{-m} u_1(\mu_0)  \leq 2^{-m}
\]
holds for all $m \in \Natural$. It follows that $u(\nu) = 0$. Since $\nu \sim \prob$ and $u$ is positively homogeneous and monotone, one can replace $\nu$ with a probability $\qprob \sim \prob$ such that $u(\qprob) = 0$, which concludes the proof of implication $(2) \Rightarrow (3)$ of Theorem \ref{thm: UI} in view of Remark \ref{rem: unif_int}.

\bibliographystyle{amsalpha} 
\bibliography{fcc}
\end{document}